\documentclass{amsart}
\usepackage[margin=3cm]{geometry}
\usepackage{amsmath,amsfonts,amssymb,amsthm}
\usepackage{graphics}
\usepackage{epsfig, esint}
\usepackage{graphics,color}
\usepackage{paralist}
\newcommand{\supp}{\mathop{\mathrm{supp}}}

\providecommand{\R}{\mathbb{R}}

\newcommand{\e}{\varepsilon}

\newcommand{\step}[1]{\medskip\noindent\textbf{Step #1. }}
\newcommand{\substep}[1]{\medskip\noindent\textit{Substep #1. }}

\newcommand{\ignore}[1]{}

\newtheorem{definition}{Definition}

\newtheorem{theorem}{Theorem}
\newtheorem{remark}{Remark}
\newtheorem{lemma}{Lemma}

\newtheorem{assumption}{Assumption}

\author[J.\ Hirsch]{Jonas Hirsch}
\address{Mathematisches Institut 
 Universit\"at Leipzig\\
 Augustusplatz 10\\ 04103 Leipzig, Germany.}
\email{jonas.hirsch@math.uni-leipzig.de}
\author[M.\ Sch\"affner]{Mathias Sch\"affner}
\address{Max-Planck-Institut f\"ur Mathematik in den Naturwissenschaften\\
 Inselstra\ss e 22\\ 04103 Leipzig, Germany.}
\email{mathias.schaeffner@mis.mpg.de}

\title[Growth conditions and regularity, an optimal local boundedness result]{Growth conditions and regularity, an optimal local boundedness result} %Local boundedness for integral functionals with $(p,q)$-growth

\begin{document}
\maketitle

% \centerline{\LARGE{NOT FOR DISTRIBUTION}}

\begin{abstract}
We prove local boundedness of local minimizers of scalar integral functionals $\int_\Omega f(x,\nabla u(x))\,dx$, $\Omega\subset\R^n$ where the integrand satisfies $(p,q)$-growth of the form
\begin{equation*}
|z|^p\lesssim f(x,z)\lesssim |z|^q+1
\end{equation*}
under the optimal relation $\frac1p-\frac1q\leq \frac1{n-1}$.
\end{abstract}

\section{Introduction and main result}

In this note, we establish a sharp local boundedness result for local minimizers of integral functionals 
\begin{equation}\label{eq:int}
\mathcal F[u,\Omega]:=\int_\Omega f(x,\nabla u)\,dx,
\end{equation}
where $\Omega\subset\R^n$, $n\geq2$, is a bounded domain and the integrand $f(x,\nabla u)$ satisfies $(p,q)$-growth of the form 
\begin{equation}\label{pqsim}
|z|^p\lesssim f(x,z)\lesssim |z|^q+1,
\end{equation}
see Assumption~\ref{ass} below. Local boundedness and H\"older continuity of local minimizer of \eqref{eq:int} in the case $1<p=q$ are classical, see the original reference \cite{GG82} or the textbook \cite{Giu}. Giaquinta \cite{G87} provided an example of an autonomous convex integrand satisfying \eqref{pqsim} with $p=2$ and $q=4$ that admits unbounded minimizer in dimension $n\geq6$. Similar examples can be found in \cite{Mar91,H92}, in particular it follows from  \cite[Section~6]{Mar91} that if
\begin{equation}\label{pq:unbounded}
q>\frac{(n-1)p}{n-1-p}=:p_{n-1}^*\quad\mbox{and}\quad 1<p<n-1,
\end{equation}
then one cannot expect local boundedness for minimizers of \eqref{eq:int} in general. In this paper we show that condition \eqref{pq:unbounded} is sharp. Before we state our main result, we recall a standard notion of local minimizer and quasi-minimizer in the context of functionals with $(p,q)$-growth
\begin{definition}
Given $Q\geq1$, we call $u\in W_{\rm loc}^{1,1}(\Omega)$ a \emph{$Q$-minimizer} of \eqref{eq:int} if and only if
\begin{equation*}
 \mathcal F(u,{\rm supp}\,\varphi)<\infty\quad\mbox{and}\quad \mathcal F(u,{\rm supp}\,\varphi)\leq Q\mathcal F(u+\varphi,{\rm supp}\,\varphi)
\end{equation*}
for any $\varphi\in W^{1,1}(\Omega)$ satisfying ${\rm supp}\;\varphi\Subset \Omega$. If $Q=1$, then $u$ is a \emph{local minimizer} of \eqref{eq:int}. Moreover, we call $u$ a \emph{quasi-minimizer} if and only if there exists $Q\geq1$ such that $u$ is a $Q$-minimizer. 
\end{definition}
%
%Next, we state the precise assumptions on the integrand $f$ in \eqref{eq:int}
%
\begin{assumption}\label{ass} Let $f:\Omega\times \R^n\to\R$ be a  Caratheodory function and suppose that $z\mapsto f(x,z)$ is convex for almost every $x\in\Omega$. Moreover, there exist $1\leq L<\infty$, $\mu\geq0$ such that for all $z\in \R^n$ and almost every $x\in\Omega$
\begin{align}
|z|^p\leq f(x,z)\leq L|z|^q+1,&\label{ass1}\\
f(x,2z)\leq \mu+Lf(x,z),\label{ass2}
\end{align}
% There exist $1\leq L<\infty$ such that the Caratheodory function $f:\Omega\times \R^d\to[0,\infty)$ satisfies for all $z\in \R^n$ and almost every $x\in\Omega$
%\begin{align}
%\nu|z|^p\leq f(x,z)\leq L(1+|z|^q),&\label{ass1}\\
%f(x,2z)\leq L(1+f(x,z))\label{ass2}
%\end{align}
\end{assumption}
Now we are in position to state the main result of the present paper 
\begin{theorem}\label{T}
Let $\Omega\subset \R^n$, $n\geq2$ and suppose Assumption~\ref{ass} is satisfied with $1\leq p\leq q<\infty$ such that
\begin{equation}\label{eq:assq:intro}
 \frac1q\geq \frac1p-\frac1{n-1}.%q\leq p_{n-1}^*:=\frac{(n-1)p}{n-1-p}\quad\mbox{if $p<n-1$}.
\end{equation}
Let $u\in W_{\rm loc}^{1,1}(\Omega)$ be a quasi-minimizer of the functional $\mathcal F$ given in \eqref{eq:int}. Then, $u\in L_{\rm loc}^{\infty}(\Omega)$. 
\end{theorem}
\begin{remark}
 We provide the proof of Theorem~\ref{T} in Section~\ref{sec:proof}. We establish a slightly more general results in which the growth condition \eqref{ass1} is replaced by
 $$
 |z|^p-g(x)^\frac{p}{p-1}\leq f(x,z)\leq L(|z|^q+g(x)^\frac{p}{p-1}),%&\label{ass1}
 $$ 
 and optimal assumptions (in the Lorentz-scale) on $g$ are imposed.
\end{remark}
Let us now relate Theorem~\ref{T} to previous results in the literature. To the best of our knowledge, the best previously known relation between $p$ and $q$ that ensures local boundedness under Assumption~\ref{ass} can be found in the paper by Fusco \& Sbordone \cite{FS90} and reads% (see also \cite{CMM15}) it is shown that under Assumption~\ref{ass} local minimizer of \eqref{eq:int} are locally bounded if the more restrictive relation 
\begin{equation}\label{pqwrong}
\frac1q\geq\frac1p-\frac1n,%q\leq \frac{pn}{n-p}=:p_n^*\quad\mbox{if $1<p<n$},
\end{equation}
see \cite[Theorem~2]{FS90} (see also the more recent result \cite[Theorem~2.3]{CMM15}) which also implies local boundedness with condition \eqref{pqwrong}). Obviously, relation \eqref{eq:assq:intro} is less restrictive than \eqref{pqwrong} and in view of the discussion above optimal for local boundedness (compare \eqref{eq:assq:intro} and \eqref{pq:unbounded}). However, we want to emphasize that \cite{CMM15,FS90} (and similarly \cite{BMS90,FS93}) contain sharp local boundedness results under additional structural assumptions on the growth of $f$, namely \textit{anisotropic growth} of the form
\begin{equation}\label{growth:aniso}
\sum_{i=1}^n|z_i|^{p_i}\lesssim f(x,z)\lesssim \sum_{i=1}^nM(1+|z_i|^{p_i}).
\end{equation}
In this case, local boundedness is proven under the condition $q\leq \overline p^*$, where $\frac1{\overline p}=\frac1n\sum_{i=1}^n\frac1{p_i}$ and $\overline p^*=\frac{n \overline p}{n-\overline p}$. This condition is optimal for local boundedness in view of the above mentioned counterexamples (the integrands in \cite{G87,Mar91,H92} satisfy growth of the form \eqref{growth:aniso}).% in the even more restrictive form $\sum_{i=1}^n|z_i|^{p_i}\lesssim f(z)\lesssim \sum_{i=1}^n(1+|z|^{p_i})$).

\bigskip

The systematic study of higher regularity of minimizers of functionals with $(p,q)$-growth was initiated by Marcellini \cite{Mar89,Mar91}. By now there is a large and quickly growing literature on regularity results for minimizers of functionals with $(p,q)$-growth, and more general non-standard growth \cite{L93,Mar93}. We refer to \cite{Min06} for an overview. A currently quite active field of research is the regularity theory for so-called double phase problems where the model functional is given by
\begin{equation}\label{int:doublephase}
\int_\Omega |\nabla u(x)|^p+a(x)|\nabla u(x)|^q\,dx,
\end{equation}
where $0\leq a\in C^{0,\alpha}$ with $\alpha\in(0,1]$ see e.g.\ \cite{BCM18,CM15,CM15b,DM19,ELM04}) and \cite{Zhikov,JKO94} for some motivation for functionals of the form \eqref{int:doublephase}. For this kind of functionals rather sharp conditions for higher ($C^{1,\beta}$-) regularity are known, where $\alpha$ has to be balanced with $p,q$, and $n$. In \cite{CM15b} it was observed that the conditions on the data can be relaxed if one a priori knows that the minimizer is bounded. Obviously by Theorem~\ref{T} the results of \cite{CM15b} can be applied without any a priori assumption whenever $\frac1q\geq\frac1p-\frac1n$ and in particular can be used to improve \cite[Theorem~5.3]{CM15b}.  Similarly, Theorem~\ref{T} improves the applicability of some results in \cite{BB18,CKP11} where also higher regularity results are proven \textit{assuming} a priori boundedness of the minimizer. 

\smallskip

Let us very briefly explain the strategy of the proof of Theorem~\ref{T} and the origin of our improvement. In principle, we use a variation of the De-Giorgi type iteration similar to e.g.\cite{FS90,FS93,CMM15}. Recall that De-Giorgi iteration is based on a Caccioppoli inequality (which yield a reverse Poincar\'e inequality) and Sobolev inequality. The main new ingredient here is to use in the Caccioppoli inequality cut-off functions that are optimized with respect to the minimizer $u$ (instead of using affine cut-offs). This enables us to use Sobolev inequality on ($(n-1)$-dimensional) spheres instead of ($n$-dimensional) balls and thus get the desired improvement. This idea, combined with a variation of Moser-iteration, was recently used by the second author and Bella in the analysis of linear non-uniformly elliptic equations \cite{BS19a} (improving in an essentially optimal way classic results of Trudinger~\cite{T71}) and for higher regularity for integral functionals with $(p,q)$-growth \cite{BS19c} (see also \cite{BS19b} for an application in stochastic homogenization).  %The trick of using certain optimized cut-off functions was previously used by the second author and Bella in \cite{BS19a,BS19c} in the analysis of linear non-uniformly elliptic equations (improving in an optimal way classic results of Trudinger~\cite{T71}) and for higher regularity for integral functionals with $(p,q)$-growth.   % compared to previous works is to an additional optimization step by choosing the cut-off function in the Caccioppoli inequality depending on the minimizer $u$. This enables us to use Sobolev-inequalities on ($n$-dimensional) spheres and by choosing a cut-off function  In principle, we use a De-Giorgi type iteration argument similar to e.g.\cite{FS90,FS93,CMM15}. The main new ingredoent but combined with an additional optimization which was prevstep by the derivation of the Cacciop   In fact, similarly to e.g.\ \cite{FS90}, we use a kind of a De Giorgi type iteration argument but with one additional

%
%By now there is a large and quickly growing literature on regularity results for minimizers of functionals with $(p,q)$-growth, and more general non-standard growth, we refer to \cite{Min06} for an overview. Under additional structural assumptions on the growth of $f$, for example anisotropic growth of the form
%%
%$$m\sum_{i=1}^n|z_i|^{p_i}\leq f(z)\leq M\sum_{i=1}^n(1+|z_i|^{q}),$$
%%
%more precise and sharp assumptions on the involved exponents that ensure higher regularity are available in the literature, see e.g.\ \cite{CMM15,FS93}. Regularity results under general structural assumptions beyond polynomial growth can be found e.g.\ in \cite{L93,Mar93}, see also the recent result \cite{EMM19} where convexity is only imposed 'at infinity'. Moreover, rather sharp conditions are known for certain non-autonomous functionals, see e.g.\ \cite{BCM18,CM15,DM19,ELM04}, where also H\"older-continuity of the integrand $f$ in the space variable has to be balanced with $p,q$, and $n$. In \cite{CKP14,ELM99} higher integrability results for autonomous integral functionals can be found that are also valid in the case of systems.

\smallskip

The paper is organized as follows: In Section~\ref{sec:prelim}, we recall some definitions and useful results regarding Lorentz spaces and present a technical lemma which is used to derive an improved version of Caccioppoli inequality which plays a prominent role in the proof of Theorem~\ref{T}. In Section~\ref{sec:proof}, we prove a slightly more general version of Theorem~\ref{T} which in particular includes some a~priori estimates. 

\section{Preliminary results}\label{sec:prelim}
\subsection{Preliminary lemmata}

A key ingredient in the prove of Theorem~\ref{T} is the following lemma which is a variation of \cite[Lemma 3]{BS19c}
\begin{lemma}\label{L:optimcutoff}
Fix $n\geq2$. For given $0<\rho<\sigma<\infty$, $v\in L^1(B_\sigma)$ and $s>1$, we consider consider
\begin{equation*}
 J(\rho,\sigma,v):=\inf\left\{\int_{B_\sigma}|v||\nabla \eta|^s\,dx \;|\;\eta\in C_0^1(B_\sigma),\,\eta\geq0,\,\eta=1\mbox{ in $B_\rho$}\right\}.
\end{equation*}
Then for every $\delta\in(0,1]$ 
\begin{equation}\label{1dmin}
 J(\rho,\sigma,v)\leq  (\sigma-\rho)^{-(s-1+\frac1\delta)} \biggl(\int_{\rho}^\sigma \left(\int_{S_r} |v|\right)^\delta\,dr\biggr)^\frac1\delta.
\end{equation}
\end{lemma}

\begin{proof}[Proof of Lemma~\ref{L:optimcutoff}]
Estimate \eqref{1dmin} follows directly by minimizing among radial symmetric cut-off functions. Indeed, we obviously have for every $\e\geq0$
\begin{equation*}
 J(\rho,\sigma,v)\leq \inf\left\{\int_{\rho}^\sigma |\eta'(r)|^s\left(\int_{S_r}|v|+\e\right)\,dr \;|\;\eta\in C^1(\rho,\sigma),\,\eta(\rho)=1,\,\eta(\sigma)=0\right\}=:J_{{\rm 1d},\e}.
\end{equation*}
For $\e>0$, the one-dimensional minimization problem $J_{{\rm 1d},\e}$ can be solved explicitly and we obtain
\begin{equation}\label{1dmin:2}
J_{{\rm 1d},\e}=\biggl(\int_{\rho}^\sigma \biggl(\int_{S_r}|v|+\e\biggr)^{-\frac1{s-1}}\,dr\biggr)^{-(s-1)}.
\end{equation}
Let us give an argument for \eqref{1dmin:2}. First we observe that using the assumption $v\in L^1(B_\sigma)$ and a simple approximation argument we can replace $\eta\in C^1(\rho,\sigma)$ with $\eta\in W^{1,\infty}(\rho,\sigma)$ in the definition of $J_{{\rm 1d},\e}$. Let $\widetilde\eta:[\rho,\sigma]\to[0,\infty)$ be given by
$$\widetilde\eta(r):=1-\biggl(\int_\rho^\sigma b(r)^{-\frac{1}{s-1}}\,dr\biggr)^{-1}\int_{\rho}^rb(r)^{-\frac{1}{s-1}}\,dr,\quad\mbox{where $b(r):=\int_{S_r}|v|+\e$}.$$
Clearly, $\widetilde \eta\in W^{1,\infty}(\rho,\sigma)$ (since $b\geq\e>0$), $\widetilde \eta(\rho)=1$, $\widetilde \eta(\sigma)=0$, and thus
\begin{equation*}
J_{{\rm 1d},\e}\leq\int_{\rho}^\sigma |\widetilde\eta'(r)|^sb(r)\,dr=\biggl(\int_{\rho}^\sigma b(r)^{-\frac{1}{s-1}}\,dr\biggr)^{-(s-1)}.
\end{equation*}
The reverse inequality follows by H\"older's inequality: For every $\eta\in W^{1,\infty}(\rho,\sigma)$ satisfying $\eta(\rho)=1$ and $\eta(\sigma)=0$, we have 
\begin{equation*}
1=\left(\int_\rho^\sigma \eta'(r)\,dr\right)^s\leq \int_{\rho}^\sigma|\eta'(r)|^sb(r)\,dr\biggl(\int_{\rho}^\sigma b(r)^{-\frac1{s-1}}\,dr\biggr)^{s-1}.
\end{equation*}
Clearly, the last two displayed formulas imply \eqref{1dmin:2}. 

Due to the monotonicity of $(-\infty,\infty)\ni m\mapsto (\fint_{\rho}^\sigma v^m(r)\,dr)^\frac1m$, we deduce from \eqref{1dmin:2} for every $\delta>0$
\begin{equation*}
J_{{\rm 1d},\e}\leq (\sigma-\rho)^{-(s-1+\frac1\delta)}\biggl(\int_{\rho}^\sigma \left(\int_{S_r}|v|+\e\right)^{\delta}\,dr\biggr)^{\frac{1}{\delta}}.
\end{equation*}
Sending $\e$ to zero, we obtain \eqref{1dmin}.
\end{proof}

In order to derive a suitable Cacciopolli type inequality in the proof of Theorem~\ref{T}, we make use of the so-called 'hole-filling' trick combined with the following useful (and well-kown) lemma 
\begin{lemma}[Lemma~6.1, \cite{Giu}]\label{L:holefilling}
Let $Z(t)$ be a bounded non-negative function in the interval $[\rho,\sigma]$. Assume that for every $\rho\leq s<t\leq \sigma$ it holds
\begin{equation*}
 Z(s)\leq \theta Z(t)+(t-s)^{-\alpha} A+B,
\end{equation*}
with $A,B\geq0$, $\alpha>0$ and $\theta\in[0,1)$. Then, there exists $c=c(\alpha,\theta)\in[1,\infty)$ such that
\begin{equation*}
 Z(s)\leq c((t-s)^{-\alpha} A+B).
\end{equation*}
\end{lemma}

\subsection{Non-increasing rearrangement and Lorentz-spaces}

We recall the definition and useful properties of the non-increasing rearrangement $f^*$ of a measurable function $f$ and Lorentz spaces, see e.g.\ \cite[Section~22]{TatarBook}. For a measurable function $f:\R^n\to\R$, the non-increasing rearrangement is defined by
\begin{equation*}
 f^*(t):=\inf\{\sigma\in(0,\infty)\,:\,|\{x\in\R^n\,:\,|f(x)|>\sigma\}|\leq t\}.
\end{equation*}
Let $f:\R^n\to\R$ be a measurable function with ${\rm supp}f\subset \Omega$, then it holds for all $p\in[1,\infty)$
\begin{equation}\label{eq:ff*}
\int_\Omega |f(x)|^p\,dx=\int_0^{|\Omega|}(f^*(t))^p\,dt.
\end{equation}
A simple consequence of \eqref{eq:ff*} and the fact $f\leq g$ implies $f^*\leq g^*$ is the following inequality
\begin{equation}\label{est:omegat}
\sup_{|A|\leq t\atop A\subset\Omega}\int_A|f(x)|^p\leq \int_0^t(f_\Omega^*(t))^p\,dt,
\end{equation}
where $f_\Omega^*$ denotes the non-increasing rearrangement of $f\chi_\Omega$ (inequality \eqref{est:omegat} is in fact an \textit{equality} but for our purpose the upper bound suffices). 

The Lorentz space $L^{n,1}(\R^d)$ can be defined as the space of measurable functions $f:\R^d\to\R$ satisfying
\begin{equation*}
 \|f\|_{L^{n,1}(\R^d)}:=\int_0^\infty t^\frac1n f^*(t)\,\frac{dt}t<\infty.
\end{equation*}
Moreover, for $\Omega\subset\R^d$ and a measurable function $f:\R^d\to\R$, we set
\begin{equation*}
 \|f\|_{L^{n,1}(\Omega)}:=\int_0^{|\Omega|} t^\frac1n f_\Omega^*(t)\,\frac{dt}t<\infty,
\end{equation*}
where $f_\Omega$ defined as above. Let us recall that $L^{n+\e}(\Omega)\subset L^{n,1}(\Omega)\subset L^n(\Omega)$ for every $\e>0$, where $L^{n,1}(\Omega)$ is the space of all measurable functions $f:\Omega\to\R$ satisfying $ \|f\|_{L^{n,1}(\Omega)}<\infty$ (here we identify $f$ with its extension by zero to $\R^n\setminus \Omega$). Following \cite[Section 9]{Kufner}, we define for given $\alpha>0$ the Lorentz-Zygmund space $L^{n,1}(\log L)^\alpha(\R^d)$ as the space of all measurable functions $f:\R^d\to\R$ satisfying%
\begin{equation*}
 \|f\|_{L^{n,1}(\log L)^\alpha(\R^d)}:=\int_0^\infty t^\frac1n (1+|\log(t)|)^\alpha f^*(t)\,\frac{dt}t<\infty.
\end{equation*}
As above, for $\Omega\subset\R^d$ and a measurable function $f:\R^d\to\R$, we set
\begin{equation*}
 \|f\|_{L^{n,1}(\log L)^\alpha(\Omega)}:=\int_0^{|\Omega|} t^\frac1n (1+|\log(t)|)^\alpha f_\Omega^*(t)\,\frac{dt}t<\infty,
\end{equation*}
and denote by $L^{n,1}(\log L)^\alpha(\Omega)$ the space of all measurable functions $f:\Omega\to\R$ satisfying $\|f\|_{L^{n,1}(\log L)^\alpha(\Omega)}<\infty$. Obviously, we have for every bounded domain $\Omega$ that $L^{n+\e}(\Omega)\subset L^{n,1}(\log L)^\alpha(\Omega)\subset L^{n,1}(\Omega)$ for every $\e>0$.

\color{black}

\section{Proof of Theorem~\ref{T}}\label{sec:proof}

In this section, we provide a proof of Theorem~\ref{T}. As mentioned in the introduction, we establish a slightly stronger statement where the growth condition \eqref{ass1} is relaxed in order to introduce a right-hand side (see Remark~\ref{rem:g} below). 
\begin{assumption}\label{assgeneral} Let $f:\Omega\times \R^n\to\R$ be a  Caratheodory function and suppose that $z\mapsto f(x,z)$ is convex for almost every $x\in\Omega$. Moreover, there exist $1\leq L<\infty$, $\mu\geq0$ such that for all $z\in \R^n$ and almost every $x\in\Omega$
\begin{align}
|z|^p-g(x)^{\frac{p}{p-1}}\leq f(x,z)\leq L|z|^q+g(x)^\frac{p}{p-1},&\label{ass1gen}\\
f(x,2z)\leq \mu+Lf(x,z)),\label{ass2gen}
\end{align}
where $g$ is a non-negative function satisfying $g\in L^{\frac{p}{p-1}}(\Omega)$.
\end{assumption}

In order to state an a priori estimate it is convenient to introduce suitable scale invariant versions of Soblev and $L^p$ norms. For any bounded domain $\Omega\subset\R^n$, we set
\begin{equation*}
 \|v\|_{\underline W^{1,p}(\Omega)}:=|\Omega|^{-\frac1n}\|v\|_{\underline L^p(\Omega)}+\|\nabla v\|_{\underline L^p(\Omega)},
\end{equation*}
where
\begin{equation*}
\|v\|_{\underline L^p(\Omega)}:=|\Omega|^{-\frac1p}\|v\|_{L^p(\Omega)}.
\end{equation*}
Note that by definition of  $\|\cdot\|_{\underline W^{1,p}(\Omega)}$, it holds
\begin{equation}\label{eq:rescaling}
\forall v\in W^{1,p}(B_R),\,R>0:\qquad \|v\|_{\underline W^{1,p}(B_R)}= \|v_R\|_{\underline W^{1,p}(B_1)}\quad\mbox{where $v_R:=\frac1Rv(R\cdot)\in W^{1,p}(B_1)$.}
\end{equation}
\begin{theorem}\label{T1}
Let $\Omega\subset \R^n$, $n\geq2$ and suppose Assumption~\ref{assgeneral} is satisfied with $1\leq p<q<\infty$ satisfying
\begin{equation}\label{eq:assq}
 \e:=\e(n,p,q):=\min\biggl\{\frac1q+\frac1{n-1},1\biggr\}-\frac1p\geq0,%q\leq p_{n-1}^*:=\frac{(n-1)p}{n-1-p}\quad\mbox{if $p<n-1$}.
\end{equation}
and suppose that %$g^\frac1{p-1}\in L^{n,1}(\Omega)$ if $p<n$ and $g^\frac1{n-1}\in L^{n+\delta}(\Omega)$ for some $\delta>0$ if $p=n$.
\begin{equation*}
 g^\frac1{p-1}\in L^{n,1}(\Omega)\quad\mbox{if $p<n$ and}\qquad g^\frac1{n-1}\in L^{n,1}(\log L)^\frac{n-1}{n}(\Omega)\quad\mbox{if $p=n$}.
\end{equation*}
\color{black}Let $u\in W_{\rm loc}^{1,1}(\Omega)$ be a quasi-minimizer of the functional $\mathcal F$ given in \eqref{eq:int}. Then, $u\in L_{\rm loc}^{\infty}(\Omega)$. Moreover, if 
\begin{equation}\label{eq:assq1}
 \e(n,p,q)>0\quad\mbox{and}\quad 1\leq p<n,%\frac1q-\biggl(\frac1p-\frac1{n-1}\biggr)=\e>0 \quad\mbox{and $1<p<n-1$},
\end{equation}
there exists $c=c(L,n,p,q,Q)\in[1,\infty)$ such that every $Q$-minimizer of \eqref{eq:int} satisfies for every $x_0\in \Omega$ with $B_R:=B_R(x_0)\Subset\Omega$ the estimate
\begin{align}\label{est:T1}
\|u\|_{L^\infty(B_\frac{R}2)}\leq c(R\|u\|_{\underline W^{1,p}(B_R)}+R\|u\|_{\underline W^{1,p}(B_R)}^{1+\frac1\e(\frac1p-\frac1n)(1-\frac{p}q)}+\|g^\frac{1}{p-1}\|_{L^{n,1}(B_R)}).
\end{align}
%\begin{align}
% \|u\|_{L^\infty(B_\frac{R}2)} \leq& c\biggl(\biggl(\fint_{B_R}u^p\biggr)^\frac1p+R\biggl(\fint_{B_R}|\nabla u|^p\biggr)^{\frac1p}\biggl(1+\biggl(\fint_{B_R}|\nabla u|^p\biggr)^{\frac1p\frac1\e(1-\frac{p}n)(\frac1p-\frac1q)}\biggr)\notag\\
% &+R\biggl(\fint_{B_R}g^s\biggr)^\frac1{sp}\biggr).
%\end{align}
\end{theorem}

\begin{remark}\label{rem:g}
As mentioned above, Theorem~\ref{T1} is optimal with respect to the relation between the exponents $p$ and $q$. Moreover, it is also optimal with respect to the assumption on $g$ (at least for $p<n$). Indeed, for $p>1$ consider 
\begin{equation}\label{f:laplaceG}
 f(x,z):= \tfrac{p+1}p |z|^p+G\cdot z,
\end{equation} 
where $G\in L^\frac{p}{p-1}(\Omega,\R^n)$. Clearly $f$ satisfies Assumption~\ref{assgeneral} with $1<p=q$, $g=\frac{p-1}p|G|$ and $L=\frac{p+2}{p}$. Note that $u\in W_{\rm loc}^{1,1}(\Omega)$ is a local minimizer of the functional $\mathcal F$ given in \eqref{eq:int} and $f$ given as in \eqref{f:laplaceG} if and only if it solves locally
\begin{equation}\label{eq:rem:g}
-\Delta_p u:=-{\rm div}(|\nabla u|^{p-2}\nabla u)=\tfrac1{p+1}{\rm div}\, G.
\end{equation}
Hence, Theorem~\ref{T1} yields local boundedness for every weak solution of \eqref{eq:rem:g} provided $|G|^\frac1{p-1}\in L^{n,1}(\Omega)$. On the contrary, the (unbounded) function $u(x)=\log(-\log(|x|))$ solves (trivially) \eqref{eq:rem:g} on $B_\frac12$ with right-hand side $G=-(p+1)|\nabla u|^{p-2}\nabla u$ 
satisfying $|G|^\frac1{p-1}=c(p)|\nabla u|$ and thus $|G|^\frac1{p-1}\in L^{n,1+\delta}(B_\frac12)$ for every $\delta>0$ (in particular $|G|^\frac1{p-1}\in L^{n,n}(B_\frac12)=L^n(B_\frac12)$) but $|G|^\frac1{p-1}\notin L^{n,1}(B_\frac12)$.
% As mentioned above, Theorem~\ref{T1} is optimal with respect to the relation between the exponents $p$ and $q$. Moreover, it is also optimal with respect to the assumption on $g$. Indeed, consider 
%%
%\begin{equation}\label{f:laplaceG}
% f(x,z):=\frac12 |z|^2+G\cdot z,
%\end{equation} 
%%
%where $G\in L^2(\Omega,\R^n)$. Clearly $f$ satisfies Assumption~\ref{assgeneral} (up to an additive constant) with $p=q=2$ and $g=2|G|$. Every local minimizer $u\in W_{\rm loc}^{1,1}(\Omega)$ of the functional $\mathcal F$ given in \eqref{eq:int} and $f$ given as in \eqref{f:laplaceG}, solves locally
%%
%\begin{equation*}
%-\Delta u={\rm div}\, G,
%\end{equation*}
%%
%and by maximal regularity, we obtain that $u$ is locally bounded if $|G|=\frac12g\in L^{n,1}(\Omega)$ (but $u$ is in general unbounded if $g\in L^{n}(\Omega)$).

In the interesting recent paper \cite{BM18}, a related result is proven on the Lipschitz-scale. More precisely it is proven that local minimizer of $\int_\Omega f(\nabla u)-gu\,dx$ are locally Lipschitz  if $f$ satisfies (controlled) $(p,q)$-growth i.e.\ 
\begin{equation}\label{pq:growthd2}
(1+|z|^2)^{\frac{p-2}{2}}|\lambda|^2\lesssim \langle D^2 f(z)\lambda,\lambda\rangle\lesssim (1+|z|^2)^\frac{q-2}2|\lambda|^2
\end{equation}
with $\frac{q}{p}<1+\frac2n$ and $g$ is in the optimal Lorentz space $L^{n,1}(\Omega)$ (provided $n\geq3$). Very recently, Lipschitz-regularity of minimizers for integrands satisfying \eqref{pq:growthd2} is proven in \cite{BS19c} under the less restrictive relation $\frac{q}{p}<1+\frac2{n-1}$ in the case $g\equiv0$. It would interesting if the methods of \cite{BS19c} and \cite{BM18} can be combined to obtain Lipschitz estimates under the assumption $\frac{q}p<1+\frac2{n-1}$ and $g\in L^{n,1}$ provided $n\geq3$. 
\end{remark}

\begin{proof}[Proof of Theorem~\ref{T1}]

By standard scaling and translation arguments it suffices to suppose that $B_1\Subset \Omega$ and prove that $u$ is locally bounded in $B_\frac12$. Hence, we suppose from now on that $B_1\Subset\Omega$. In Step~1--Step~3 below, we consider the case $p\in[1,n)$. We first derive a suitable Caccioppoli-type inequality (Step~1) and perform a De Giorgi-type iteration (Step~2 and 3) to prove boundedness from above for a $Q$-miniminzer. In Step~4, we discuss how this implies the claim of the theorem in the case $p\in[1,n)$. In Step~5, we sketch the adjustments for the remaining non-trivial case $p=n$.

\step 1 Basic energy estimate.

We claim that there exists $c=c(n,p,q,Q)\in[1,\infty)$ such that for every $k\geq0$ and every $\frac12\leq \rho<\sigma\leq1$ it holds
\begin{align}\label{est:substep11}
 \|\nabla (u-k)_+\|_{L^p(B_\rho)}^p\leq& c\biggl(\omega(|A_{k,\sigma}|)+\frac{L|A_{k,\sigma}|^{q\e}}{(\sigma-\rho)^{\gamma}} \|(u-k)_+\|_{W^{1,p}(B_\sigma)}^q\biggr),
\end{align}
where $\gamma=\gamma(n,q):=q-1+q\min\{\frac1q+\frac1{n-1},1\}$, $\e$ as in \eqref{eq:assq}, 
\begin{equation}
A_{l,r}:=B_r\cap \{x\in\Omega\, :\,u(x)>l\}\qquad\mbox{for all $r>0$ and $l>0$,}
\end{equation}
and $\omega:[0,|B_1|]\to[0,\infty)$ is defined by
\begin{equation}\label{def:omegat}
 \omega(t):=\int_0^t ((g^\frac1{p-1}\chi_{B_1})^*(t))^p\,dt.
\end{equation}
%
%with the understanding $(g^\frac1{p-1})^*:=(g^\frac1{p-1}\chi_{B_1})^*$.\\

%\begin{equation}
% \omega(t)\geq\sup_{|A|\leq t\atop A\subset B_1}\int_{A}g^p(x)\,dx\qquad\mbox{and}\qquad \lim_{t\downarrow0}\omega(t)=0.
%\end{equation}
%
Fix $M>k$ and let $\eta\in C_c^1(B_1,[0,1])$ be such that $\eta=1$ in $B_\rho$ and $\supp \eta\subset B_\sigma$. Define $w:=\max\{u_M-k,0\}$ where $u_M:=\min\{u,M\}$ and set $\varphi:=-\eta^qw$. Since $u$ is a quasi-minimizer, we obtain with help of  convexity of $f$ that 
\begin{align*}
\int_{A_{k,\sigma}}f(x,\nabla u(x))\,dx\leq& Q\int_{A_{k,\sigma}}f(x,\nabla (u+\varphi)(x))\,dx\\
=&Q\int_{A_{k,\sigma}\cap \{u\leq M\}}f(x,(1-\eta^q)\nabla u-q\eta^{q-1}\nabla \eta(u_M-k)_+)\,dx\\
&+Q\int_{A_{k,\sigma}\cap \{u> M\}}f(x,\nabla u+q\eta^{q-1}\nabla \eta(-(u_M-k)_+)\,dx\\
\leq&Q\int_{A_{k,\sigma}\cap \{u\leq M\}}(1-\eta^q)f(x,\nabla u)+\eta^qf(x,-\frac{q\nabla \eta}{\eta}(u_M-k)_+)\,dx\\
&+\frac{Q}2\int_{A_{k,\sigma}\cap \{u> M\}}f(x,2\nabla u)+f(x,-2q\eta^{q-1}\nabla \eta(u_M-k)_+)\,dx
\end{align*}
and thus, using \eqref{ass1gen}, \eqref{ass2gen} and $|\eta|\leq1$,
\begin{align}\label{est:cacc1}
\int_{A_{k,\sigma}}f(x,\nabla u(x))\,dx\leq&Q\int_{A_{k,\sigma}\setminus B_\rho}f(x,\nabla u)\,dx+\frac{Q}2\int_{A_{k,\sigma}\cap \{u> M\}}\mu+Lf(x,\nabla u)\,dx\notag\\
&+ Q \int_{A_{k,\sigma}}g^\frac{p}{p-1}+Lq^q(1+2^{q})|\nabla \eta|^q|(u_M-k)_+|^q\,dx.
\end{align}
We claim that there exists $c=c(n,q)\in[1,\infty)$ such that
\begin{align}\label{est:infcutoff}
 \inf_{\eta\in \mathcal A(\rho,\sigma)}\int_{A_{k,\sigma}}|\nabla \eta|^q|(u_M-k)_+|^q\leq& c(\sigma-\rho)^{-\gamma}\|(u-k)_+\|_{W^{1,p}(B_\sigma\setminus B_\rho)}^q|A_{k,\sigma}|^{q\e},
 \end{align}
 where $\mathcal A(\rho,\sigma):=\{\eta\in C_c^1(B_\sigma),\, \eta\equiv1\mbox{ on $B_\rho$}\}$. To show \eqref{est:infcutoff}, we use the Sobolev inequality on spheres, i.e\  there exists $c=c(n,q)\in[1,\infty)$ such that for every $r>0$
\begin{equation}\label{ineq:sobd3}
\biggl(\int_{S_r}|\varphi|^{q}\biggr)^\frac1{q}\leq c\biggl(\biggl(\int_{S_r}|\nabla \varphi|^{q_*}\biggr)^\frac1{q_*}+\frac1r\biggl(\int_{S_r}|\varphi|^{q_*}\biggr)^\frac1{q_*}\biggr),
\end{equation}
where $q_*\geq1$ is given by $\frac1{q_*}=\min\{\frac1q+\frac1{n-1},1\}$. Combining \eqref{ineq:sobd3} applied to $\varphi=(u-k)_+$ and Lemma~\ref{L:optimcutoff} with $\delta:=\frac{q_*}q>0$ yield
\begin{align*}%\label{est:infcutoff1plarge}
 &\inf_{\eta\in \mathcal A(\rho,\sigma)}\int_{A_{k,\sigma}}|\nabla \eta|^q|(u_M-k)_+|^q\\
 \leq& (\sigma-\rho)^{-(q-1+\frac{q}{q_*})}\biggl(\int_{\rho}^\sigma \biggl(\int_{S_r}|(u-k)_+|^{q}\biggr)^\frac{q_*}{q}\biggr)^\frac{q}{q_*}\\
 \leq& c(\sigma-\rho)^{-\gamma}\biggl(\int_{\rho}^\sigma  \biggl[\left(\int_{S_r} |\nabla(u-k)_+|^{q_*}\right)+\left(\int_{S_r} |(u-k)_+|^{q_*}\right)\biggr]\,dr\biggr)^\frac{q}{q_*}
 \end{align*}
(note that we  ignored the factor $\frac1r$ in \eqref{ineq:sobd3} in view of $\frac12\leq\rho<\sigma\leq1$). Finally, we observe that $\e\geq0$ implies that $q_*\leq p$ and we obtain with help of H\"older inequality
\begin{align*}%\label{est:infcutoff1plarge}
 \inf_{\eta\in \mathcal A(\rho,\sigma)}\int_{A_{k,\sigma}}|\nabla \eta|^q|(u_M-k)_+|^q\leq&c(\sigma-\rho)^{-\gamma}\|(u-k)_+\|_{W^{1,q_*}(B_\sigma\setminus B_\rho)}^q\\
 \leq& c(\sigma-\rho)^{-\gamma}\|(u-k)_+\|_{W^{1,p}(B_\sigma\setminus B_\rho)}^q|A_{k,\sigma}|^{q\e},
 \end{align*}
and \eqref{est:infcutoff} is proven.

Since \eqref{est:cacc1} is valid for all $\eta\in \mathcal A(\rho,\sigma)$, we deduce from \eqref{est:infcutoff}, \eqref{est:omegat} and $f(x,z)\geq -g(x)^\frac{p}{p-1}$
\begin{align}\label{est:cacc2}
\int_{A_{k,\rho}}f(x,\nabla u)\,dx \leq&Q\int_{A_{k,\sigma}\setminus B_\rho}f(x,\nabla u)\,dx+\frac{Q}2\int_{A_{k,\sigma}\cap \{u> M\}}\mu+Lf(x,\nabla u)\,dx\notag\\
&+ (Q+1)\omega(|A_{k,\sigma}|)+\frac{cLQ}{(\sigma-\rho)^{\gamma}}\|(u-k)_+\|_{W^{1,p}(B_\sigma)}^q|A_{k,\sigma}|^{q\e},
\end{align}
where $c=c(n,p,q)\in[1,\infty)$ and $\omega$ is defined in \eqref{def:omegat}. Since $u$ is a quasi-minimizer and we assume $B_1\Subset\Omega$, we have that $f(\cdot,\nabla u)\in L^1(B_1)$ and $u\in W^{1,1}(B_1)$. Thus, we can send $M\to\infty$ in \eqref{est:cacc2} and the second term on the right-hand side in \eqref{est:cacc2} vanishes. Hence, we obtain with help of the hole-filling trick (namely adding $Q\int_{A_{k,\rho}}f(x,\nabla u)\,dx $ to both sides of inequality \eqref{est:cacc2})
\begin{align*}
\int_{A_{k,\rho}}f(x,\nabla u(x))\,dx\leq \theta \int_{A_{k,\sigma}}f(x,\nabla u(x))\,dx+c\biggl(\omega(|A_{k,\sigma}|)+\frac{L|A_{k,\sigma}|^{q\e}}{(\sigma-\rho)^{\gamma}}\|(u-k)_+\|_{W^{1,p}(B_\sigma)}^q\biggr),
\end{align*}
with $\theta=\frac{Q}{Q+1}\in[0,1)$ and $c=c(n,p,q)\in[1,\infty)$. Estimate \eqref{est:substep11} follows by Lemma~\ref{L:holefilling} and \eqref{ass1gen}.

\step 2 One-step improvement.

We claim that there exist $c_1=c_1(n,p,q,Q)\in[1,\infty)$ and $c_2=c_2(n,p)\in[1,\infty)$ such that for every $k>h\geq 0$ and every $\frac12\leq \rho<\sigma<1$ it holds
\begin{align}\label{est:onestep}
J(k,\rho)\leq& c_1\biggl(\omega\biggl(\frac{c_2 J(h,\sigma)^\frac{p_n^*}p}{(k-h)^{p_n^*}}\biggr)+L\biggl(\frac{J(h,\sigma)^\frac1p}{k-h}\biggr)^{p_n^*q\e}\frac{J(h,\sigma)^\frac{q}p}{(\sigma-\rho)^{\gamma}}+\biggl(\frac{J(h,\sigma)^{\frac1p}}{(k-h)}\biggr)^{p_n^*\frac{p}n}J(h,\sigma)\biggr),
\end{align}
where $p_n^*:=\frac{p n}{n-p}$ and for any $l\geq0$ and $r>0$
\begin{equation*}
J(l,r):=\|(u-l)_+\|_{W^{1,p}(B_r)}^p.%\theta\int_{A_{k,\rho}}f(x,\nabla u(x))\,dx+c\biggl(\biggr)
\end{equation*}
Note that $k-h<u-h$ on $A_{k,r}$ for every $r>0$ and thus with help of Sobolev inequality
\begin{equation}\label{est:Aksigma}
|A_{k,\sigma}|\leq \int_{A_{k,\sigma}}\biggl(\frac{u(x)-h}{k-h}\biggr)^{p_n^*}\leq \frac{\|(u-h)_+\|_{L^{p_n^*}(B_\sigma)}^{p_n^*}}{(k-h)^{p_n^*}}\leq c \frac{J(h,\sigma)^\frac{p_n^*}p}{(k-h)^{p_n^*}},%\|(u-h)_+\|_{W^{1,p}(B_\sigma)}^{p_n^*},%\int_{A_{h,\sigma}}|(u-h)_+|^{p_n^*}\,dx\leq \frac{c\|(u-h)_+\|_{W^{1,p}(B_\sigma)}^{p_n^*}}{(k-h)^{p_n^*}}\leq c\frac{\|\nabla (u-h)_+\|_{L^{p}(B_\sigma)}^{p_n^*}}{(k-h)^{p_n^*}}.
\end{equation}
where $c=c(n,p)\in[1,\infty)$. Combining the above estimate with \eqref{est:substep11}, we obtain
\begin{equation}\label{est:onestep1}
\|\nabla (u-k)_+\|_{L^p(B_\rho)}^p\leq c_1\biggl(\omega\biggl(\frac{c_2 J(h,\sigma)^\frac{p_n^*}p}{(k-h)^{p_n^*}}\biggr)+L\biggl(\frac{J(h,\sigma)^\frac1p}{k-h}\biggr)^{p_n^*q\e}\frac{J(h,\sigma)^\frac{q}p}{(\sigma-\rho)^{\gamma}}\biggr),
\end{equation}
where $c_1=c_1(n,p,q,Q)\in[1,\infty)$ and $c_2=c_2(n,p)\in[1,\infty)$. It is left to estimate $\|(u-k)_+\|_{L^p(B_\rho)}$. A combination of H\"older inequality, Sobolev inequality and estimate \eqref{est:Aksigma} yield
\begin{align}\label{est:onestep2}
 \| (u-k)_+\|_{L^p(B_\rho)}^p\leq \|(u-h)_+\|_{L^{p_n^*}(B_\sigma)}^p|A_{k,\sigma}|^{\frac{p}{n}}\leq c \biggl(\frac{J(h,\sigma)^{\frac1p}}{(k-h)}\biggr)^{p_n^*\frac{p}n}J(h,\sigma)\biggr)
\end{align}
Combining \eqref{est:onestep1} and \eqref{est:onestep2}, we obtain \eqref{est:onestep}.

\step 3 Iteration.

For $k_0\geq0$ and a sequence $(\Delta_\ell)_{\ell\in\mathbb N}\subset [0,\infty)$ specified below, we set
\begin{equation}
 k_\ell:=k_0+\Delta_\ell,\quad \sigma_\ell=\frac12+\frac1{2^{\ell+1}}.
\end{equation}
For every $\ell\in\mathbb N\cup\{0\}$, we set $J_\ell:=J(k_\ell,\sigma_\ell)$. From \eqref{est:onestep}, we deduce for every $\ell\in\mathbb N$
\begin{align}\label{est:iteration}
J_\ell\leq c_1\biggl(\omega\biggl(\frac{c_2 J_{\ell-1}^\frac{p_n^*}p}{(\Delta_\ell)^{p_n^*}}\biggr)+L2^{(\ell+1)\gamma}\biggl(\frac{J_{\ell-1}^\frac1p}{\Delta_\ell}\biggr)^{p_n^*q\e}J_{\ell-1}^\frac{q}p+\biggl(\frac{J_{\ell-1}^{\frac1p}}{\Delta_\ell}\biggr)^{p_n^*\frac{p}n}J_{\ell-1}\biggr),
\end{align}
where $c_1$ and $c_2$ are as in Step~2. Fix $\tau=\tau(n,p,q)\in(0,1)$ such that
\begin{equation}\label{def:tau}
2^\gamma\tau^{\frac{q}p(1+p_n^{*}\e)-1}=\frac12.
\end{equation}
%
%Suppose that $\omega$ satisfies
%%
%\begin{equation}\label{ass:omega1}
% \int_0^{|B_1|}\frac{\omega(s)^\frac1p}{s^{\frac1p-\frac1n}}\frac{\omega'(s)}{\omega(s)}\,ds<\infty
%\end{equation}
We claim that we can choose $\{\Delta_\ell\}_{\ell\in\mathbb N}$ satisfying
\begin{equation}\label{est:sumdeltaell}
\sum_{\ell\in\mathbb N}\Delta_\ell<\infty
\end{equation}
and $k_0$ (in the borderline case $\e=0$) in such a way that
\begin{equation}\label{ass:iterationJell}
 J_\ell\leq \tau^\ell J_0\qquad\mbox{for all $\ell\in\mathbb N\cup\{0\}$}.
\end{equation}

\substep{3.1} Suppose that $\e>0$. Set $k_0=0$ and choose $\Delta_\ell$ to be the smallest number such that
\begin{align}\label{def:deltaell1}
 c_1\omega\biggl(\frac{c_2 (\tau^{\ell-1}J_0)^\frac{p_n^*}p}{(\Delta_\ell)^{p_n^*}}\biggr)\leq \frac13\tau^\ell J_0,\qquad c_1\tau^{-(\frac{p_n^*}p+1)}J_0^{\frac{p_n^*}n}\tau^{\ell\frac{p_n}p}\leq \frac13\Delta_\ell^{p_n^*\frac{p}n}
\end{align}
and
\begin{equation}\label{def:deltaell2}
c_1L2^\gamma\tau^{-\frac{q}p(1+p_n^*\e)}J_0^{\frac{q}p(1+p_n^*\e)-1}2^{-\ell}\leq \frac13\Delta_\ell^{p_n^*q\e}
\end{equation}
is valid. The choice of $\tau$ (see \eqref{def:tau}), $\Delta_\ell$ and estimate \eqref{est:iteration} combined with a straightforward induction argument yield \eqref{ass:iterationJell}. Using $\sum_{\ell\in\mathbb N}(2^{-\alpha}+\tau^\beta)<\infty$ for any $\alpha,\beta>0$, we deduce from \eqref{def:deltaell1} and \eqref{def:deltaell2}
\begin{align}\label{sum:deltaell}
 \sum_{\ell\in\mathbb N}\Delta_\ell \leq \sum_{\ell\in\mathbb N}\frac{c_2^{\frac1p-\frac1n}(\tau^{\ell-1}J_0)^\frac1p)}{(\omega^{-1}(\frac{\tau^\ell J_0}{3c_1}))^{\frac1p-\frac1n}}+c (J_0^\frac1p+J_0^{\frac1p+\frac1p(\frac1p-\frac1n)(1-\frac{p}q)\frac1\e}),
\end{align}
where $c=c(L,n,p,q,Q)\in[1,\infty)$. Next, we show that $g^\frac1{p-1}\in L^{n,1}(B_1)$ ensures that the first term on the right-hand side of \eqref{sum:deltaell} is bounded and thus \eqref{est:sumdeltaell} is valid. Indeed,
\begin{align}\label{sum:omega1}
\sum_{\ell\in\mathbb N}\frac{(\tau^{\ell}J_0)^\frac1p)}{(\omega^{-1}(\frac{\tau^\ell J_0}{3c_1}))^{\frac1p-\frac1n}}\lesssim& \int_1^\infty \frac{(\tau^x J_0)^\frac1p}{(\omega^{-1}(\frac{\tau^xJ_0}{3c_1}))^{\frac1p-\frac1n}}\,dx\notag\\
=&\frac1{|\log \tau|}\int_0^\tau\frac{(tJ_0)^\frac1p}{(\omega^{-1}(\frac{tJ_0}{3c_1}))^{\frac1p-\frac1n}}\,\frac{dt}t\notag\\
\leq&\frac{(3c_1)^\frac1p}{|\log\tau|}\int_0^{\omega^{-1}({\frac{\tau J_0}{3c_1}})} \frac{( \omega(s))^\frac1p}{s^{\frac1p-\frac1n}}\frac{\omega'(s)}{\omega(s)}\,ds.
\end{align}
Recall $\omega(t)=\int_0^t((g^\frac1{p-1}\chi_{B_1})^*(s))^p\,ds$ and $(g^\frac1{p-1}\chi_{B_1})^*$ is non-increasing, thus $\omega(t)\geq t(g^\frac{1}{p-1}\chi_{B_1})^*(t)^p$ and 
\begin{align}\label{sum:omega2}
\int_0^{\omega^{-1}({\frac{\tau J_0}{3c_1}})} \frac{( \omega(s))^\frac1p}{s^{\frac1p-\frac1n}}\frac{\omega'(s)}{\omega(s)}\,ds\leq&\int_0^{\infty}s^{\frac1n-\frac1p}(s((g^\frac1{p-1}\chi_{B_1})^*(s))^p)^{-(1-\frac1p)}((g^\frac1{p-1}\chi_{B_1})^*(s))^p\,ds\notag\\
=&\int_0^{\infty}s^{\frac1n}(g^\frac1{p-1}\chi_{B_1})^*(s)\,\frac{ds}s=\|g^\frac1{p-1}\|_{L^{n,1}(B_1)}.
\end{align}
Notice that \eqref{ass:iterationJell} and $k_0=0$ implies
\begin{equation*}
\|(u-\sum_{\ell\in\mathbb N}\Delta_\ell)_+\|_{L^p(B_\frac12)}=0\quad \Rightarrow\quad \sup_{B_\frac12}u\leq \sum_{\ell\in\mathbb N}\Delta_\ell
\end{equation*}
and thus
$$
\sup_{B_\frac12}u\leq \sum_{\ell\in\mathbb N}\Delta_\ell.
$$
Hence, appealing to \eqref{sum:deltaell}-\eqref{sum:omega2}, we find $c=c(L,n,p,q,Q)\in[1,\infty)$ such that
\begin{equation}\label{est:supu}
 \sup_{B_\frac12}u\leq c(\|(u)_+\|_{W^{1,p}(B_1)}+\|(u)_+\|_{W^{1,p}(B_1)}^{1+(\frac1p-\frac1n)(1-\frac{p}q)\frac1\e}+\|g\|_{L^{n,1}(B_1)}).
\end{equation}
%
%\begin{align*}
%J_\ell:=J(k_\ell,\sigma_\ell)^p\leq c\frac{2^{\ell (p_n^*+\gamma)}}{k^{(p_n^*(1-\frac{q}{p_{n-1}^*}))}}\biggl((\frac{1}k)^{p_n^*\frac{q}{p_{n-1}^*}}+J_0^{\frac{q}{p}(1-\frac{p_n^*}{p_{n-1}^*})}\biggr)J_{\ell-1}^\frac{p_n^*}{p}.
%\end{align*}
%%

\substep{3.1} Suppose that $\e=0$. We claim that
\begin{equation}\label{lim:J00}
 \lim_{k_0\to\infty} J_0=0.
\end{equation}
Before, we give the argument for \eqref{lim:J00} we explain how \eqref{lim:J00} implies the desired claim \eqref{ass:iterationJell} in the case $\e=0$. Choose $\Delta_\ell$ to be the smallest number such that \eqref{def:deltaell1} is satisfied and choose $k_0$ sufficiently large such that 
\begin{equation*}%\label{est:J0small}
c_1L2^\gamma \tau^{-\frac{q}p}J_0^{\frac{q}p-1}\leq\frac13.
\end{equation*}
It is now easy to see that the choice of $\tau$, $\Delta_\ell$, $k_0$ and estimate \eqref{est:iteration} yield \eqref{ass:iterationJell}. In view of Substep~3.1 we also have $\sum_{\ell\in\mathbb N}\Delta_\ell<\infty$ and we have
\begin{equation*}
 \sup_{B_\frac12}u\leq k_0+c(\|(u)_+\|_{W^{1,p}(B_1)}+\|g\|_{L^{n,1}(B_1)})<\infty,
\end{equation*}
where $c=c(L,n,p,q,Q)\in[1,\infty)$.

Let us now show \eqref{lim:J00}. For $k\geq 2^\frac1p|B_1|^{-\frac1p}\|u\|_{L^p(B_1)}$ we have $|A_{k,1}|\leq\frac12 |B_1|$ and thus a suitable version of Poincare inequality (see e.g.\ \cite[Proposition~3.15]{GM12}) yields
\begin{equation*}
 \int_{B_1}|(u-k)_+|^p\,dx\leq c\int_{B_1}|\nabla (u-k)_+|^p\,dx,
\end{equation*}
where $c=c(n,p)\in[1,\infty)$. Hence, it suffices to show $\lim_{k\to\infty}\|\nabla (u-k)_+\|_{L^p(B_1)}^p=0$. By \eqref{ass1gen}, we have for every $k\geq0$
\begin{align}\label{lim:J001}
 \int_{B_1}|\nabla (u-k)_+|^p=\int_{A_{k,1}}|\nabla u|^p\leq \int_{A_{k,1}}f(x,\nabla u)+g^\frac{p}{p-1}(x)\,dx.
\end{align}
Since $B_1\Subset \Omega$ and $f(x,\nabla u),g^\frac{p}{p-1}\in L^1_{\rm loc}(\Omega)$, the right-hand side in \eqref{lim:J001} tends to zero as $k$ tends to infinity and thus \eqref{lim:J00} is proven.

\step 4 Conclusion in the case $p<n$.% for the case $p\in[1,n-1)$. 

In view of Step~1--Step~3, we have that if $B_1\Subset \Omega$ than $u$ is locally bounded from above in $B_\frac12$ and in the case $\e>0$, we have the estimate \eqref{est:supu}. Moreover, if $u$ is a $Q$-minimizer of $\mathcal F$, then $-u$ is a $Q$-minimizer of the functional $\widetilde{\mathcal F}(v):=\int_{\Omega} \tilde f(x,\nabla v(x))\,dx$ with $\tilde f(x,z):=f(x,-z)$. Clearly, $\tilde f$ is convex in the second component and satisfies the same growth conditions as $f$. Hence, we obtain that $u$ is locally bounded in $B_\frac12$. Moreover, if $\e>0$ there exists $c=c(L,n,p,q,Q)\in[1,\infty)$ such that
\begin{equation*}
 \|u\|_{L^\infty(B_\frac12)}\leq c(\|u\|_{W^{1,p}(B_1)}+\|u\|_{W^{1,p}(B_1)}^{1+(\frac1p-\frac1n)(1-\frac{p}q)\frac1\e}+\|g\|_{L^{n,1}(B_1)}).
\end{equation*}
The conclusion of the theorem in the case $p\in[1,n)$ now follows by standard scaling, translation and covering arguments (here we use \eqref{eq:rescaling}).

\step 5 The case $p=n$.

We use the same notation as in the previous steps and sketch the necessary adjustments. Note that for $p=n$ we cannot use Sobolev inequality in the form \eqref{est:Aksigma}. In the parts not involving $\omega$ it suffices to replace $p_n^*$ by any $\tilde p\in[q,\infty)$ (recall $q>p=n$) and we leave the details to the reader. Using this replacement for the estimates related to $\omega$, we obtain local boundedness under slightly stronger assumptions on $g$, namely $g^\frac1{n-1}\in L^{n+\delta}(\Omega)$ for some $\delta>0$ (in fact this statement is already contained in \cite{FS93}). Thus we may appeal to the Moser-Trudinger inequality, which gives for some dimensional constant $c>0$, $0\le h <k , \frac12<\sigma <1$
\begin{equation}\label{est:Aksigmapn}
|A_{k,\sigma}|\leq c\exp\biggl(-\frac1c\left(\frac{k-h}{J(h,\sigma)^\frac1n}\right)^{\frac{n}{n-1}}\biggr).
\end{equation}
Let us first conclude and present the the derivation of the above inequality below. 

In view of Step 3 and \eqref{est:Aksigmapn} it suffices to show that the sequence $\{\Delta_\ell\}_{\ell\in\mathbb N}$ defined by the identity
\[\omega\biggl(c\exp\biggl(-\frac1c\frac{\Delta^{\frac{n}{n-1}}_\ell}{(\tau^{\ell-1} J_0)^\frac{1}{n-1}}\biggr)\biggr)= \bar c\tau^\ell J_0,\]
for some $\bar c>0$ and $\tau\in(0,1)$ is summable. Indeed, we have
\begin{align*}
\sum_{\ell\in\mathbb N}\Delta_\ell\lesssim& \sum_{\ell\in\mathbb N}(\tau^{\ell-1} J_0)^\frac1n\left(1 +  |\log(\tfrac1{c}\omega^{-1}(\bar c\tau^\ell J_0))|\right)^{\frac{n-1}{n}}\\
\lesssim& \frac1{\tau^\frac1n|\log(\tau)|}\int_0^\tau(tJ_0)^\frac1n(1+|\log(\omega^{-1}(\bar ctJ_0))|)^{\frac{n-1}{n}}\,\frac{dt}t\\
=& \frac{1}{(\tau\bar c)^\frac1n|\log(\tau)|}\int_0^{\omega^{-1}(\bar c\tau J_0)}(\omega(s))^\frac1n(1+|\log(s)|)^{\frac{n-1}{n}}\frac{\omega'(s)}{\omega(s)}\,ds.
%\leq& \frac{c_1}{|\log(\tau)|}\int_0^{\frac1{c_1}\omega^{-1}(c_3\tau J_0)}(\tfrac1{c_3}c_1s)^\frac1n (g^\frac1{n-1}\chi_{B_1})^*(c_1s)|\log(s)|\frac{\omega'(c_1s)}{\omega(c_1s)}\,ds
\end{align*}
Now we can continue as before, i.e. using $\omega(s)\geq s(g^\frac{1}{n-1}\chi_{B_1})^*(s)^n$ and $\omega'(s)=(g^\frac{1}{n-1}\chi_{B_1})^*(s)^n$, we obtain
\begin{align*}%\label{sum:omega2}
\int_0^{\omega^{-1}(\bar c\tau J_0)}(\omega(s))^\frac1n(1+|\log(s)|)^{\frac{n-1}{n}}\frac{\omega'(s)}{\omega(s)}\,ds\leq&\int_0^{\infty}s^{\frac1n}(1+|\log(s)|)^{\frac{n-1}{n}}((g^\frac1{n-1}\chi_{B_1})^*(s))\,\frac{ds}s\notag\\
=&\|g^\frac1{n-1}\|_{L^{n,1}(\log L)^{\frac{n-1}{n}}(B_1)}<\infty.
\end{align*}
Finally, we present the argument for \eqref{est:Aksigmapn}. For this we recall the Moser-Trudinger inequality in the following form: there exists $c_i=c_i(n)>0$, $i=1,2$ such that for every ball $B\subset\R^d$ and every $v\in W^{1,n}(B)$
\begin{equation*}
 \fint_{B}\exp\biggl(\biggl(\frac{|v-\fint_Bv|}{c_1\|\nabla v\|_{L^n(B)}}\biggr)^\frac{n}{n-1}\biggr)\leq c_2
\end{equation*}
(see e.g.\ \cite[Chapter 7]{GT}).
Since 
\[A_{k, \sigma} \subset B_\sigma \cap \{ x \colon  (u-h)_+ \ge k-h\}=:E_{h,k,\sigma}\]
Chebychev's inequality combined with Moser-Trudinger inequality  gives \eqref{est:Aksigmapn}:
\begin{align*}
|E_{h,k,\sigma}|
\lesssim& \exp\biggl(-\left(\frac{k-h}{2^{\frac{n}{n-1}}c_1 J(h,\sigma)^\frac1n}\right)^\frac{n}{n-1}\biggr) \int_{B_\sigma} \exp\biggl(\biggl(\frac{(u-h)_+}{ 2^{\frac{n}{n-1}}c_1 J(h,\sigma)^\frac1n}\biggr)^\frac{n}{n-1}\biggr)\,dx\\
\leq&\exp\biggl(-\left(\frac{k-h}{2^{\frac{n}{n-1}}c_1 J(h,\sigma)^\frac1n}\right)^\frac{n}{n-1}\biggr) \exp\biggl(\biggl(\frac{\fint_{B_\sigma}(u-h)_+}{ c_1 J(h,\sigma)^\frac1n}\biggr)^\frac{n}{n-1}\biggr)c_2|B_\sigma|\\
\lesssim&\exp\biggl(-\left(\frac{k-h}{2^{\frac{n}{n-1}}c_1 J(h,\sigma)^\frac1n}\right)^\frac{n}{n-1}\biggr),%\int_{B_\sigma} \exp\biggl(\biggl(\frac{|(u-h)_+-\fint_{B_\sigma}(u-h)_+}{ c_1 J(h,\sigma)^\frac1n}\biggr)^\frac{n}{n-1}\biggr)\,dx
\end{align*}
where we use in the last estimate the Poincar\'e inequality the assumption $\sigma\leq1$.

\end{proof}

\section*{Acknowledgments}

M.S. was supported by the German Science Foundation DFG in context of the Emmy Noether Junior Research Group BE 5922/1-1. J.H. was supported by the German Science Foundation DFG in context of the Priority Program SPP 2026 "Geometry at Infinity".
%\appendix

%\section{Proof of Lemma~\ref{L:Mar}}

\end{document}